\documentclass[12pt]{article}
\usepackage{amsmath}
\usepackage{latexsym}
\usepackage{amsfonts}
\usepackage{amssymb}
\usepackage{amscd}
\usepackage{theorem}

\usepackage{color}

\newtheorem{theorem}{Theorem}[section]

\newtheorem{lemma}[theorem]{Lemma}
\newtheorem{proposition}[theorem]{Proposition}
\newtheorem{corollary}[theorem]{Corollary}

{\theorembodyfont{\rmfamily}

  \newtheorem{remark}[theorem]{Remark}
}

\newenvironment{proof}{    
  \noindent
  \textbf{Proof.}}{
  \hfill $\Box$
  \vspace{3mm}
}

\numberwithin{equation}{section}

\newcommand{\N}{\mathbb{N}} 
\newcommand{\R}{\mathbb{R}} 
\newcommand{\C}{\mathbb{C}} 
\newcommand{\D}{\mathbb{D}} 

\begin{document}

\title{Solid hulls of weighted Banach spaces of entire functions.  }

\author{Jos\'{e} Bonet  and Jari Taskinen }

\date{}

\maketitle

\begin{abstract}
Given a continuous, radial, rapidly decreasing  weight $v$ on the complex plane, we study the solid hull of its associated weighted space $H_v^\infty(\C)$ of all the entire functions $f$ such that $v|f|$ is bounded. The solid hull is found for a large
class of weights satisfying the condition (B) of Lusky. Precise formulations are
obtained for weights of the form $v(r)=\exp(-ar^p), a>0, p>0$. Applications to spaces of multipliers are included. 
\end{abstract}

\renewcommand{\thefootnote}{}

\section{Introduction and first results.} \label{sec1}

The aim of this paper is to investigate the solid hull of weighted Banach spaces  $H_v^\infty(\C)$ of all
entire functions $f$ such that $\Vert f \Vert_v:= \sup_{z \in \C} v(z)|f(z)|$ is finite. In what follows, we identify an entire function $f(z)=\sum_{n=0}^\infty a_n z^n$ with the sequence of its Taylor coefficients $(a_n)_{n=0}^\infty$. For example in the case $v(z) = e^{-|z|}, z \in \C,$ we show in
Theorem \ref{solidhull1} that the solid hull
consists precisely of complex sequences $(b_m)_{m=0}^\infty$ such that
\begin{equation}
\sup_{n \in \mathbb{N}}
\sum_{m=n^2 + 1 }^{(n+1)^2} |b_m|^2
e^{-2n^2} n^{4m}  
< \infty . \nonumber
\end{equation}
We are also able to characterize in Theorem \ref{solidhull}, the solid hulls for a quite general class of  weights in terms of numerical sequences defined by  Lusky, \cite{L},
in his investigations of the isomorphic classes of the spaces  $H_v^\infty(\C)$.
This class of weights includes those satisfying condition (B) of \cite{L}, see
Remark \ref{remLB} and Corollary \ref{corex}.
The calculation of the numerical sequences for some important weights $v$ is one of the
results of our paper, see Proposition \ref{exampleexponential}.
In addition to techniques of \cite{L}, our approach  uses
the methods of Bennet, Stegenga and Timoney in their paper \cite{BST}, where
the solid hull and the solid core of the weighted spaces $H^\infty_v(\D)$ were
determined for doubling weights $v$ on the open unit disc $\D$. In Section \ref{sec4} we show that our results in Section \ref{solid} can be used to determine space of multipliers from $H_v^\infty(\C)$ into $\ell^p, 1 \leq p \leq \infty$.

The solid hull and multipliers on spaces of analytic functions on the disc has been investigated by many authors. In addition to  \cite{BST}, we mention here a non exhaustive sample:  \cite{AS}, \cite{Bl}, \cite{BRV}, \cite{D}, \cite{JP}, \cite{V} and the list of references in \cite{BP}. Moreover, the papers \cite{BG}, \cite{Li} and \cite{Tung} investigate the behavior of the Taylor coefficients of entire functions belonging to weighted spaces similar to those considered in this paper.
Spaces of type $H_v^\infty(\C)$ and $H^\infty_v(\D)$ appear
in the study of growth conditions of
analytic functions and have been investigated in various articles since
the work of Shields and Williams, see {\it e.g.} \cite{BBG},\cite{BBT}, \cite{BT}, \cite{L0},
\cite{L}, \cite{SW} and the references therein.

A {\it weight} $v$   is a continuous function  $v: [0, \infty[ \to ]0,  \infty [$, which is non-increasing on $[0,\infty[$ and it is rapidly decreasing, i.e.\ it satisfies
$\lim_{r \rightarrow \infty} r^m v(r)=0$ for each $m \in \N$. We extend $v$ to $\C$ by $v(z):= v(|z|)$. For
such a weight, the {\it weighted Banach space of entire functions}
is defined by
\begin{center}
$H^\infty_v(\C) := \{ f \in H(\C) \, : \,  \Vert f \Vert_v := \sup_{z \in \C} v(|z|) |f(z)| <
 \infty \}$,
\end{center}
and it is  endowed with the weighted sup norm $\Vert \cdot  \Vert_v .$  Spaces of this type are also called sometimes \textit{weighted Fock spaces of infinite order}. For an entire function $f \in H(\C)$, we denote $M(f,r):= \max\{|f(z)| \ | \ |z|=r\}$. Using the notation $O$ and $o$ of Landau, $f \in H_v^\infty(\C)$ if and only if $M(f,r)=O(1/v(r)), r \rightarrow \infty$. The symbol $\N$ stands for the natural numbers $n=1,2,3,...$.

As we already mentioned, an entire function $f(z)=\sum_{n=0}^\infty a_n z^n$ is identified with the sequence of its Taylor coefficients $(a_n)_{n=0}^\infty$, that will be also denoted sometimes by $(a_n)_n$.
As is well-known, it is often impossible to characterize standard Banach spaces of
entire functions in terms of the Taylor coefficients; this is for example true
for the function spaces $H_v^\infty (\C)$.  The next best thing is to find the
strongest growth condition that the coefficients have to satisfy. This motivates
the concept of a solid hull, and we now recall the related
definitions and facts from \cite{AS}.

Let $A$ and $B$ be vector spaces of complex sequences containing the space of all the sequences with finitely many non-zero coordinates. The space $A$ is \textit{solid} if $a=(a_n) \in A$ and $|b_n| \leq |a_n|$ for each $n$ implies $b=(b_n) \in A$. The \textit{solid hull of $A$} is

$$S(A):= \{ (c_n) \, : \, \exists (a_n) \in A \ \mbox{such that} \ |c_n| \leq |a_n| \ \forall n \in \N \}. $$

\noindent The \textit{solid core of $A$} is

$$s(A):=\{ (c_n) \, : \, (c_na_n) \in A \ \forall (a_n) \in \ell_\infty \}.$$

\noindent The \textit{set of multipliers form $A$ into $B$} is

$$(A,B):= \{ c=(c_n) \, : \, (c_na_n) \in B \ \forall (a_n) \in A \}. $$

\medskip

\noindent \textbf{Facts:} 1. $A$ is solid if and only if $\ell_\infty \subset (A,A)$.

\noindent 2. $A \subset (B,C)$ if and only if $B \subset (A,C)$.

\noindent 3. The solid core $s(A)$ of $A$ is the largest solid space contained in $A$. Moreover $s(A)=(\ell_\infty,A)$.

\noindent 4. The solid hull $S(A)$ of $A$ is the smallest solid space containing $A$.

\noindent 5. If $X$ is solid, $(A,X)=(S(A),X)$ and $(X,A)=(X,s(A))$.

\medskip

We conclude this section with our first results.

\begin{proposition}\label{core}
The solid core of $H^\infty_v(\C)$ is
$$
s(v,\C):=\{ (a_n)_n \, : \, \Vert(a_n)\Vert_{1,v,\C} := \sup_{r>0} v(r) \sum_{n=0}^\infty |a_n| r^n < \infty \}.
$$
\end{proposition}
\begin{proof}
Given $(a_n)_n \in s(v,\C)$, the function $f(z)= \sum_{n=0}^\infty a_n z^n$ is clearly entire.
Moreover $v(|z|)|f(z)| \leq \Vert(a_n)_n\Vert_{1,v,\C}$ for each $z \in \C$, and $f \in H^\infty_v(\C)$.

To see the other inclusion, let $A$ be a solid sequence space contained in $H^\infty_v(\C)$, i.e.\ for each $(a_n)_n \in A$, $(|a_n|)_n \in A$ and $g(z)= \sum_{n=0}^\infty |a_n| z^n \in H^\infty_v(\C)$. Clearly, $M(g,r)= \sup_{|z|=r} \big| \sum_{n=0}^\infty |a_n| z^n \big| = \sum_{n=0}^\infty |a_n| r^n$. Therefore
$$\sup_{r>0} v(r) \sum_{n=0}^\infty |a_n| r^n = \sup_{r>0} v(r)M(g,r) < \infty.$$
\end{proof}

\begin{remark} Write for $n \in \N$, $\Vert z^n \Vert_v:= \sup_{r>0} v(r)r^n$. Clearly
$\Vert z^0\Vert_v=v(0)$. The weighted $\ell_1$ space $\ell_1((\Vert z^n\Vert_v)$ of all those complex sequences $(a_n)_n$ such that $\sum_{n=0}^\infty |a_n| \Vert z^n\Vert_v < \infty$ is contained in the solid core $s(v,\C)$. However, in general the inclusion is strict. Indeed, if $v(r)=e^{-r}, r \geq 0$, then $\Vert z^n\Vert_v = (n/e)^n$ as a direct calculation shows. Take $a_n:=1/n!, \ n=0,1,2,...$. For each $r>0$, we have $v(r) \sum_{n=0}^\infty  a_n r^n =1$. However, the series $\sum_{n=0}^\infty (n^n/n! e^n)$ diverges by the Stirling's formula $n! \sim (2 \pi n)^{1/2} (n/e)^n$.
\end{remark}

Our next elementary result about the behavior of the Taylor coefficients of elements $f \in H_v^\infty(\C)$, that holds for arbitrary weights $v$,  clarifies the importance of the study of the solid hull of $H_v^\infty(\C)$ in Section \ref{solid}.

\begin{proposition}\label{hull}
The Taylor coefficients of an entire function $f(z)=\sum_{n=0}^\infty a_n z^n$ in $H^\infty_v(\C)$ satisfy

\begin{itemize}

\item[(i)] $\sup_{r>0} v(r) \left(\sum_{n=0}^\infty |a_n|^2 r^{2n} \right)^{1/2} \leq \Vert f\Vert_v$, and

\item[(ii)] $\sup_{n} |a_n| \Vert z^n\Vert_v \leq \Vert f\Vert_v$.

\end{itemize}
\end{proposition}
\begin{proof}
$(i)$ Since $H^\infty$ is contained in $H^2$, for each function $h(z)= \sum_{n=0}^\infty b_n z^n$ analytic in a neighbourhood of $\D$, we have $(\sum_{n=0}^\infty |b_n|^2)^{1/2} \leq M(h,1)$. Now, given
$f(z)=\sum_{n=0}^\infty a_n z^n \in H^\infty_v(\C)$ and $r>0$, set $g(\zeta):=f(r \zeta)$, that is clearly an analytic function in a neighbourhood of $\D$. Since $g(\zeta)=f(r \zeta) = \sum_{n=0}^\infty a_n r^n \zeta^n$, we conclude
$$
\left(\sum_{n=0}^\infty |a_n|^2 r^{2n} \right)^{1/2} \leq M(g,1) =M(f,r).
$$
This implies the inequality in the statement (i).

(ii) follows from (i).
\end{proof}

\begin{remark} In general it is not true that $f \in H^\infty_v(\C)$ implies that $(|a_n| \Vert z^n\Vert_v)_n \in \ell_2$. To see this, take again $v(r)=e^{-r}, r \geq 0$. Clearly $e^z = \sum_{n=0}^\infty (1/n!) z^n$ belongs to $H^\infty_v(\C)$. However, the series $\sum_{n=0}^\infty (|a_n| \Vert z^n\Vert_v)^2 = \sum_{n=0}^\infty (n^n/n! e^n)^2$ diverges because $(n^n/n! e^n)^2 \sim 1/n$ by Stirling's formula.
\end{remark}

\section{The solid hull of $H^\infty_v(\C)$.}\label{solid}

We fix for this section a weight $v: \C \to ]0,\infty[$ satisfying the general hypothesis
made in Section \ref{sec1}. Our next aim is to   characterize  the solid hull of $H^\infty_v(\C)$ for weights
satisfying the additional condition \eqref{35}, below.
Let us start by introducing some notation used in \cite{L}.

We denote by $[x]$ be the largest
integer less or equal $x$ for a given real number $x \in \R$.
Given $m>0$, we denote by $r_m$ the global maximum point of $r^m v(r)$. Then $r_m \rightarrow \infty$ as $m \rightarrow \infty$. For example, if $v(r)= \exp(-\alpha r^p)$, then $r_m=(m/\alpha p)^{1/p}, \ m>0$.
Given an entire function $f(z)= \sum_{k=0}^\infty a_k z^k$, and $0<m<n$ (not necessarily integers) we define the following operators of de la Vall\'ee-Poussin type:
$$
V_{n,m}f:= \sum_{0 \leq k \leq m } a_k z^k + \sum_{m < k \leq n} \frac{[n]-k}{[n]-[m]} a_k z^k,
$$
$$
V_{p,0}f:= \sum_{0 \leq k \leq p}  \frac{[p]-k}{[p]} a_k z^k \, ;
$$
here and later,  the summation is performed over integers belonging to
the given intervals, although the endpoints of the intervals need not be
integers. We also denote, for $0<m<n$,
\begin{equation}
A(m,n):= \left(\frac{r_{m}}{r_{n}}\right)^{m} \frac{v(r_{m})}{v(r_{n})} \ \
\ \mbox{and}  \ \ \
B(m,n):= \left(\frac{r_{n}}{r_{m}}\right)^{n} \frac{v(r_{n})}{v(r_{m})}.
\label{28}
\end{equation}

Several results of Lusky, \cite{L}, will be needed below.
We start with the following lemma.

\begin{lemma} \label{luskylemma1}
(\cite{L}, Corollary 3.2 (b))
Let $0<m<n$ and let $Q(z)=\sum_{m < k \leq n} a_k z^k$ be a polynomial. Then
$$ \Vert Q\Vert_v \leq 2 A(m,n) \sup_{|z|=r_n} |Q(z)| v(z),$$
and
$$ \Vert Q\Vert_v \leq 2 B(m,n) \sup_{|z|=r_m} |Q(z)| v(z).$$

\end{lemma}

Given a strictly increasing sequence $(m_n)_{n=1}^{\infty}$ with $\lim_{n \rightarrow \infty} m_n =
\infty$, we define $m_0=0$ and $V_n := V_{m_{n+1},m_n} - V_{m_{n},m_{n-1}},
n \in \N$. For each $n \in \N, n \geq 2,$
$V_nf$ is a polynomial with all terms of degree  at least $m_{n-1} + 1 $ and
at most $m_{n+1}$, for all $f \in H(\C)$. In fact, for each $f(z)=\sum_{n=0}^\infty f_n z^n\in H(\C)$, we have
$$
V_nf(z)= \sum_{m_{n-1} < m \leq m_{n+1}} \gamma_m f_m z^m
= \sum_{m= [m_{n-1}] + 1}^{[ m_{n+1}]} \gamma_m f_m z^m ,
$$
where the numbers $\gamma_m \in [0,1]$  are
$$
\gamma_m = \frac{[m_{n+1}]-m}{[m_{n+1}]-[m_n]}, \ \ m_n < m \leq m_{n+1}
$$
and
$$
\gamma_m = \frac{m-[m_{n-1}]}{[m_{n}]-[m_{n-1}]}, \ \ m_{n-1} < m \leq m_{n}.
$$
The sum in $V_1f$ is understood to go from $m=0$ to $m=m_1$.

\begin{lemma}\label{luskylemma2} (\cite{L}, Prop. 3.4 (b))
If the sequence $(m_n)_n$ with $\lim_{n \rightarrow \infty} m_n = \infty$ satisfies
$$
\min(A(m_n,m_{n+1}),B(m_n,m_{n+1})) \geq b
$$
for some $b>2$, then there is $D>0$ such that $\Vert V_n\Vert  \leq D$ for each $n \in \N$, where $\Vert .\Vert $ is the operator norm in $L(H^\infty_v(\C))$ with respect to $\Vert \cdot \Vert_v$.
\end{lemma}

\begin{lemma} \label{luskylemma3} (\cite{L}, Lemma 5.1)
Fix $b>1$. For each weight $v: [0, \infty[ \rightarrow ]0, \infty[$ there is a sequence of numbers $0<m_1 <m_2<...$ with $\lim_{n \rightarrow \infty} m_n = \infty$, such that
$$
A(m_n,m_{n+1}) \geq b \ \ \ {\rm and} \ \ \ B(m_n,m_{n+1}) \geq b
$$
and
$$
\sup\limits_{n \in \N} \min(A(m_n,m_{n+1}),B(m_n,m_{n+1})) < \infty.
$$
\end{lemma}

In fact, Lusky proves in \cite[Lemma 5.1]{L} that $(m_n)_n$ can be selected to satisfy
$\min(A(m_n,m_{n+1}),B(m_n,m_{n+1}))=b$ for each $n \in \N$.
Our next lemma is contained in  Proposition 5.2 of \cite{L}, and its proof in \cite{L}
only uses the fact that $\min(A(m_n,m_{n+1}),B(m_n,m_{n+1})) \geq b$, not that
$\min(A(m_n,m_{n+1}),$ $B(m_n,m_{n+1}))=b$ for each $n \in \N$.
(In Remark \ref{rem9} we will be  able to clarify, which of these quantities $A$ or $B$
is larger in the case of the weights $\exp(-ar^p)$.)
Accordingly, the following result holds true, and it contains an important expression for a
norm equivalent to that of $H_v^\infty(\C)$.

\begin{lemma} \label{luskylemma4} (\cite{L}, Proposition 5.2)
Let the numbers $(m_n)_{n=1}^\infty$ and $b$ be as in Lemma
\ref{luskylemma3}, and in addition assume that  $b > 2$. Then, for every
$f \in H_v^\infty(\C)$,
\begin{eqnarray}
& & c_1 \sup\limits_{n \in \N} \sup\limits_{r_{m_{n-1}} \leq |z| \leq r_{m_{n+1}}}
|f_n(z)| v(z)
\leq
\Vert f \Vert_v \leq
c_2 \sup\limits_{n \in \N} \sup\limits_{r_{m_{n-1}} \leq |z| \leq r_{m_{n+1}}}
|f_n(z)| v(z) , \nonumber
\end{eqnarray}
where $f_n := V_n f$.

\end{lemma}

The main result of this section reads as follows.

\begin{theorem}
\label{solidhull}
Let $v : \mathbb{C} \to ]0,\infty[$ be a   weight
(which is continuous,  radial, non-increasing on $[0,\infty[$, and rapidly decreasing).
In addition, assume that there exists a sequence $0<m_1 <m_2<...$ with
$\lim_{n \rightarrow \infty} m_n = \infty$  such that for some $b>2$ and some $K \geq b$ we have
\begin{equation}
b \leq \min(A(m_n,m_{n+1}),B(m_n,m_{n+1})) \leq \max(A(m_n,m_{n+1}),B(m_n,m_{n+1})) \leq K,
\label{35}
\end{equation}
for each $n \in \N$. Then, the solid hull of $H^\infty_v(\C)$ is
\begin{equation}
S(v,\C):= \Big\{ (b_m)_{m=0}^\infty \, : \,
\sup_{n \in \N} v(r_{m_{n}})
\Big( \sum_{m_n < m \leq m_{n+1}} |b_m|^2
 r_{m_n}^{2m} \Big)^{1/2} < \infty \  \Big\} .  \label{37}
\end{equation}
\end{theorem}

\begin{remark} Notice that if a weight $v$, as in Theorem \ref{solidhull}, satisfies \eqref{35}, then
the conclusion of Lemma \ref{luskylemma4} also holds.  In fact,
many important weights do satisfy 
\eqref{35}, see Corollary \ref{corex} and Remark \ref{remLB}. Moreover, examples
with explicit calculations of the sequence $(m_n)_n$ will be presented in
Section 3.

Of course, there is no need to extend the condition \eqref{37} to the finitely many coefficients
$b_m$, $0 \leq m \leq m_1$.

In the characterization \eqref{37} we could as well replace
$r_{m_n}$  by $r_{m_{n+1}}$. This follows from the proof below, or as well from
the the right hand side inequality in the assumption \eqref{35} together with the definition \eqref{28}.
Namely, the condition \eqref{35} implies
\begin{equation}
\frac{1}{K} \leq \Big( \frac{r_{m_{n+1}}}{r_{m_n}} \Big)^{m_n}
\frac{v (r_{m_{n+1}})}{v(r_{m_n})}
\leq \Big( \frac{r_{m_{n+1}}}{r_{m_n}} \Big)^{m_{n+1}}
\frac{v (r_{m_{n+1}})}{v(r_{m_n})} \leq K.
\end{equation}
Hence, for every $m$ with $m_n \leq m \leq m_{n+1}$ we have
\begin{equation}
\frac{1}{K} \leq \Big( \frac{r_{m_{n+1}}}{r_{m_n}} \Big)^{m}
\frac{v (r_{m_{n+1}})}{v(r_{m_n})}  \leq K,
\end{equation}
or, for all $m_n \leq m \leq m_{n+1}$,
\begin{equation}
\frac{1}{K} r_{m_n}^{m} v(r_{m_n}) \leq r_{m_{n+1}}^{m} v (r_{m_{n+1}})
  \leq K r_{m_n}^{m} v(r_{m_n}).
\end{equation}
This means that $ r_{m_n}^{m} v(r_{m_n})$ may be replaced by
$ r_{m_{n+1}}^{m} v (r_{m_{n+1}})$ in \eqref{37}.
\end{remark}

\begin{proof}
The proof will be obtained in two steps.

\noindent
\textbf{Step 1.} \textit{If $g(z) = \sum_{m=0}^\infty b_m z^m\in H_v^\infty(\C)$, and $\Vert g\Vert_v \leq 1$, then there is $C>0$ such that
$$
\Big( \sum_{m_n < m \leq  m_{n+1}} |b_m|^2
 r_{m_n}^{2m} \Big)^{1/2}
\leq  \frac{C}{v(r_{m_{n}})}
$$
for all $ n \in\N$.}

This step is the analogue in our setting of Theorem 1.8 in \cite{BST}.

We estimate first the sum from $m_n$ to $(m_n + m_{n+1})/2$. To do this, observe that if
$m_n \leq m \leq (m_n + m_{n+1})/2$, then $\gamma_m \geq 1/2$.
We have
\begin{eqnarray}
& & \Big( \sum_{m=[m_n]+1}^{[(m_n + m_{n+1})/2]} |b_m|^2  r_{m_n}^{2m}  \Big)^{1/2}
\leq
2 \Big( \sum_{m=[m_n]+1}^{[(m_n + m_{n+1})/2]}
\gamma_m^2 |b_m|^2  r_{m_n}^{2m}  \Big)^{1/2}
\nonumber \\
& \leq &
2 \Big( \sum_{m= [m_{n-1}] +1}^{[m_{n+1}]} \gamma_m^2 |b_m|^2  r_{m_n}^{2m}  \Big)^{1/2}
= 2 \Big( \int\limits_{\partial \D}
  \Big| \sum_{ m_{n-1} < m \leq  m_{n+1} } \gamma_m b_m  r_{m_n}^{m} z^m \Big|^2 dz
 \Big)^{1/2}
\nonumber \\
& \leq &
 4 \pi  \sup\limits_{|z| = 1}
\Big| \sum_{m_{n-1} < m \leq  m_{n+1} } \gamma_m b_m  r_{m_n}^{m} z^m \Big|
= 4 \pi  \sup\limits_{|z| = r_{m_n} }
\Big| \sum_{m_{n-1} < m \leq  m_{n+1} } \gamma_m b_m  r_{m_n}^{m}
\Big( \frac{z}{ r_{m_n}}  \Big)^m  \Big|
\nonumber \\
& = &
\frac{4 \pi}{v(r_{m_{n}})}  v(r_{m_{n}}) \sup\limits_{|z| = r_{m_n} }
\Big| \sum_{m_{n-1} < m \leq  m_{n+1}} \gamma_m b_m  z^m  \Big|
\nonumber \\
& = &   \frac{4 \pi}{v(r_{m_{n}})}
v(r_{m_{n}}) \sup\limits_{|z| = r_{m_n} }
\Big| V_n g (z)   \Big|
\leq   \frac{4 \pi \Vert V_n g \Vert_v }{v(r_{m_{n}})}
\leq \frac{4 \pi D}{v(r_{m_{n}})} ,
\end{eqnarray}
since the operators $V_n$ are uniformly bounded with respect to $\Vert \cdot  \Vert_v$ by $D>0$, by Lemma \ref{luskylemma2} and \eqref{35}.

Now we estimate the sum from $(m_n + m_{n+1})/2 $ to $m_{n+1}$.
Observe that
$$
V_{n+1}g(z)= \sum_{m_{n} < m \leq m_{n+2}} \tilde\gamma_m b_m z^m
= \sum_{m = [m_{n}] + 1 }^{[ m_{n+2}]} \tilde\gamma_m b_m z^m,
$$
where the numbers $\tilde\gamma_m \in [0,1]$  are
$$
\tilde\gamma_m = \frac{m-[m_{n}]}{[m_{n+1}]- [m_{n}]}, \ \ m_{n} < m \leq m_{n+1},
$$
which increase from $\tilde\gamma_{m_n+1}$ till $\tilde\gamma_{m_{n+1}}=1$. If $(m_n + m_{n+1})/2 < m \leq
m_{n+1}$, we have $\tilde\gamma_m \geq 1/2$. Thus, proceeding similarly as we did before, we get
\begin{eqnarray}
& & \Big( \sum_{m=[(m_n + m_{n+1})/2]+1}^{[m_{n+1}]} |b_m|^2  r_{m_n}^{2m}  \Big)^{1/2}
\leq
2 \Big( \sum_{m=[(m_n + m_{n+1})/2]+1}^{[m_{n+1}]}
 \tilde\gamma_m^2 |b_m|^2  r_{m_n}^{2m}  \Big)^{1/2}
\nonumber \\
& \leq &
2 \Big( \sum_{m=[m_{n}]+1}^{[m_{n+2}]} \tilde\gamma_m^2 |b_m|^2  r_{m_n}^{2m}  \Big)^{1/2}
\nonumber \\
& \leq&
\frac{4 \pi}{v(r_{m_{n}})}
v(r_{m_{n}}) \sup\limits_{|z| = r_{m_n} }
\Big| V_{n+1} g (z)   \Big|
\leq   \frac{4 \pi \Vert V_{n+1} g \Vert_v }{v(r_{m_{n}})}
\leq \frac{4 \pi D}{v(r_{m_{n}})}.
\end{eqnarray}
This completes the proof of Step 1. \\

Observe that the estimates proved in Step 1 remain valid if we replace $r_{m_n}$
by any  $r_m$ with $m_n \leq m \leq m_{n+1}$. \\

\noindent
\textbf{Step 2.} \textit{For each $(b_m)_{m=0}^\infty \in S(v,\C)$ there is  $f(z)= \sum_{m=0}^\infty a_m z^m  \in H_v^\infty(\C)$ such that $|b_m| \leq |a_m|$ for each $m=0,1,2,...$.}

Fix $C_0 > \sup_{n} v(r_{m_{n}})
\Big( \sum_{m_n < m \leq m_{n+1}} |b_m|^2
 r_{m_n}^{2m} \Big)^{1/2}$.

For all $n  \in\N$ and
$m_n <  m \leq m_{n+1}$,  we apply \cite{BST}, Corollary 2.8  (which is a consequence of a deep result of Kisliakov \cite{Ki}), to the sequence
$$
\big( b_{m  }  r_{m_n}^{m} \big)_{m_n < m \leq m_{n+1}  },
$$
to choose a polynomial
$$
P_n (z) = \sum_{m_n < m \leq m_{n+1} } b_m' z^m
$$
such that
$$
|b'_m| \geq |b_m|  r_{m_n}^{m} \ \forall m \  \ \ \  \mbox{and}
$$
\begin{equation}
\sup\limits_{|z| < 1 } | P_n (z)| = \sup\limits_{|z| = 1 } | P_n (z)|\leq B
\Big( \sum_{m_n < m \leq m_{n+1} } |b_m |^2  r_{m_n}^{2m} \Big)^{1/2}  .
\label{45}
\end{equation}
Here  $B>0$ is an absolute constant. Define
$$
Q_n (z) := \sum_{m_n < m \leq m_{n+1} }  r_{m_n}^{- m}  b_m' z^m
$$
and
$$
g(z) =  \sum_{n=1}^\infty  Q_n(z) =  \sum_{n=1}^\infty
\sum_{m_n < m \leq m_{n+1} } r_{m_n}^{-m}  b_m' z^m  .
$$
We still have to show that $g$ is a well defined entire function and that $g \in H_v^\infty(\C)$. However, observe that if $a_m$ denotes the $m$-th Taylor coefficient of $g$, then $|a_m| \geq |b_m|$ for all $m$.

\medskip

By Lemma \ref{luskylemma1}, \eqref{35} and \eqref{45}, we have
\begin{eqnarray}
& &
\Vert Q_n \Vert_v 
\leq 2K \sup\limits_{|z| = r_{m_n} } |  Q_n  (z) | v(z)
=  2K  \sup\limits_{|z| = r_{m_n} } \Big| \sum_{m_n < m \leq m_{n+1}}
r_{m_n}^{- m} b_m' z^m  \Big|
v(z)
\nonumber \\
&=&
2K \sup\limits_{|z| = r_{m_n} }  \Big| \sum_{m_n < m \leq m_{n+1} } b_m'
( z / r_{m_n})^m  \Big|  v(z)
=  2K v ( r_{m_n} ) \sup\limits_{|z| = 1}
\Big| \sum_{m_n < m \leq m_{n+1} } b_m'   z^m  \Big|
\nonumber \\
&=&
2K v ( r_{m_n} ) \sup\limits_{|z| < 1}
|P_n (z)|
\nonumber \\
&\leq &
2KB v( r_{m_n}) \Big( \sum_{m_n < m \leq m_{n+1}} r_{m_n}^{2m} |b_m|^2 \Big)^{1/2}
 \leq 2KBC_0, \nonumber
\end{eqnarray}
for each $n \in \N$.
Moreover, for all  $n \geq 2$, we have
\begin{eqnarray}
& &
\sup\limits_{r_{m_{n-1}} \leq |z| \leq r_{m_{n+1}}} |V_n  (Q_n + Q_{n-1} ) (z)|v(z)
\leq
\sup\limits_{z \in \C} |V_n  (Q_n + Q_{n-1} ) (z)|v(z)
\nonumber \\
&\leq &
D \sup\limits_{z \in \C}  |( Q_n + Q_{n-1})  (z)| v(z)
\leq D \Vert Q_n \Vert_v +  D \Vert Q_{n-1} \Vert_v \leq 4KBC_0D.
\label{47a}
\end{eqnarray}
Consequently for each $N >n \geq 2$, we have $V_n(\sum_{j=2}^N Q_j)(z)= V_n (Q_n + Q_{n-1})(z), z \in \C$.
By  Lemma  \ref{luskylemma4}     and \eqref{47a}
we get for every  $N \in \N$,
\begin{eqnarray}
& &
\Vert \sum_{j=2}^N Q_j \Vert_v = \sup\limits_{z \in \C} v(z)
\Big|\sum_{j=2}^N Q_j(z) \Big|
\nonumber \\
&\leq &
c_2 \sup\limits_{n \in \N} \sup\limits_{r_{m_{n-1}} \leq  \atop
|z| \leq r_{m_{n+1}}}
\Big|V_n \Big( \sum_{j=2}^N Q_j(z) \Big) \Big|v(z)
\nonumber \\
&= &
c_2 \sup\limits_{n \in \N}
 \sup\limits_{r_{m_{n-1}} \leq \atop |z| \leq r_{m_{n+1}}}
|V_n  (Q_n + Q_{n-1} ) (z)|v(z)
\leq
4KBC_0 c_2 D.
\end{eqnarray}
This implies that the sequence of polynomials $\left( \sum_{j=2}^N Q_j\right)_N$ is
contained in a multiple of the unit ball of $H_v^\infty(\C)$, which is compact for the
compact open topology. Accordingly, there is a subsequence $\left( \sum_{j=2}^{N(s)}
Q_j\right)_s$ converging to $h \in H_v^\infty(\C)$ for this topology. Since the operator of
$k$-th differentiation is continuous for the compact open topology, it follows that the
Taylor coefficients of $h$ and $g-Q_1$ coincide. This implies that $g$ is an entire function
and that $g \in H_v^\infty(\C)$. The proof of Step 2 is now complete by taking the function $f(z):=g(z)+\sum_{0 \leq m \leq m_1} b_m z^m$.
\end{proof}

\begin{remark}
\label{remLB}
(1) Lusky introduces the following \textit{condition (B)} on the weight $v$ in \cite{L}:
$$
\forall b_1>0 \ \exists b_2 >1 \ \exists c>0 \ \forall m,n :
$$
$$
\left(\frac{r_m}{r_n}\right)^m \frac{v(r_m)}{v(r_n)} \leq b_1 \ \ {\rm and} \ \ |m-n| \geq c \Rightarrow \left(\frac{r_n}{r_m}\right)^n \frac{v(r_n)}{v(r_m)} \leq b_2.
$$
By \cite{L}, Theorem 1.1, if $v$ has condition (B), then $H_v^\infty(\C)$ is isomorphic to $\ell_\infty$, and if $v$ does not satisfy condition (B), then $H_v^\infty(\C)$ is isomorphic to $H^\infty$.

We show that if the weight $v$ satisfies condition (B), then for each $b>2$ one can find a sequence $(m_n)_n$ satisfying the assumption (\ref{35}) in Theorem \ref{solidhull}. Indeed, let $v$ be a weight satisfying condition (B). Given $b>2$ we apply Lemma \ref{luskylemma3} to find a sequence $0<m_1 <m_2<...$ with $\lim_{n \rightarrow \infty} m_n = \infty$, such that
$A(m_n,m_{n+1}) \geq b, B(m_n,m_{n+1}) \geq b$ and $M:= \sup_n \min(A(m_n,m_{n+1}),B(m_n,m_{n+1})) < \infty$. An inspection of the proof of Lemma 5.1 in Lusky shows that we can take in our Lemma \ref{luskylemma3} the sequence $(m_n)_n$ such that $\lim_{n \rightarrow \infty} (m_{n+1}-m_n) = \infty$. Set $b_1 := M$, and select $b_2>1$ and $c>0$ according to condition (B). There is $n(0) \in \N$ such that $m_{n+1}-m_n \geq c$ if $n \geq n(0)$. Condition (B) now implies that $\max(A(m_n,m_{n+1}),B(m_n,m_{n+1})) \leq \max(M,b_2)$ for each $n \geq n(0)$. The proof is complete if we take $(m_n)_{n \geq n(0)}$. \\

(2) Lusky constructs in \cite{L}, Example 2.6, a weight $v$ on $\C$ not satisfying condition (B) such that for a certain sequence $(m_n)$ with $m_{n+1}-m_n = n+1$, $A(m_n,m_{n+1})=(n+1)^{n+1}$ and $B(m_n,m_{n+1})=1$ for each $n \in \N$.
\end{remark}

This remark and Examples 2.1--2.2 in \cite{L} imply the following result.

\begin{corollary}
\label{corex}
Condition \eqref{35} is satisfied by the following weights:

\noindent  $v(r)= \exp(-r^p)$ with  $p > 0$,
$v(r)= \exp( -\exp r)$,  and
$v(r)= \exp\big(- (\log^+ r  )^p\big) $, where $p \geq 2$ and $\log^+ r
= \max(\log r,0)  $.
\end{corollary}

\section{Examples.}

In this section we calculate the sequences $(m_n)_{n=1}^\infty$ for the weights
$v(r) = \exp ( -ar^p) $ and thus obtain satisfactory representations of the
corresponding solid hulls.

\begin{theorem}
\label{solidhull1}
Let $v $ be the weight $v(r) = \exp ( -ar^p) $ on $\mathbb{C}$, where $a> 0$ and $p>0$ are
constants. Then,  the solid hull of $H^\infty_v(\C)$ is
\begin{equation}
\Big\{ (b_m)_{m=0}^\infty \, : \,
\sup_{n \in \mathbb{N}}
 \sum_{pn^2 + 1 < m \leq p(n+1)^2} |b_m|^2
e^{ -2 n^2} n^{4m/p} (ap)^{-m/p}  < \infty \  \Big\} .  \label{55}
\end{equation}
\end{theorem}

In particular, the solid hull for $v(r)= \exp (-r) $ is
$$
\Big\{ (b_m)_{m=0}^\infty \, : \,
\sup_{n \in \mathbb{N}}
\sum_{m=n^2 + 1 }^{(n+1)^2} |b_m|^2
e^{-2n^2} n^{4m}  
< \infty \  \Big\} .
$$

Theorem \ref{solidhull1} is an immediate consequence of Theorem \ref{solidhull}
and the following proposition, where we choose $b = e$. The proposition gives
the Lusky numbers  $(m_n)_n$ for a class of important weights.

\begin{proposition}\label{exampleexponential}
Let $v(r)= \exp(-a r^p)$, $a>0$, $p>0$ and let $b>2$. The sequence $m_n:= p (\log b) n^2$,
$ n \in \N$, satisfies, for each $n \in \N$, $n\geq 4$,
$$
b \leq A(m_n,m_{n+1})= \left(\frac{r_{m_n}}{r_{m_{n+1}}}\right)^{m_n} \frac{v(r_{m_n})}{v(r_{m_{n+1}})} \leq b^{9/2}
$$
and
$$
b \leq B(m_n,m_{n+1})= \left(\frac{r_{m_{n+1}}}{r_{m_n}}\right)^{m_{n+1}} \frac{v(r_{m_{n+1}})}{v(r_{m_n})} \leq b^{4}.
$$
\end{proposition}

Proposition \ref{exampleexponential} implies that the numbers $m_n$
can chosen to be $n^2$ for the weight
$v(r) = \exp(-r)$ and $2 n^2$ for the weight $v(r) = \exp(-r^2)$.
We give the proof of this proposition in several steps.

\begin{lemma}\label{log1}
If $0<m<M$ satisfies $M<2m$, then
$$
\exp\left(\frac{1}{2}\frac{(M-m)^2}{M}\right) \leq \left(\frac{M}{m}\right)^M e^{m-M} \leq
\exp\left(\frac{(M-m)^2}{M}\right).
$$

\end{lemma}
\begin{proof} If $0<x<1/2$, then
\begin{equation}
 x^2 / 2 \leq - \log(1-x) -x \leq x^2 \label{65}
\end{equation}
This is so since
$$
-\log(1-x)-x= x^2/2 + x^3/3 +... \geq x^2/2
$$
and on the other hand,
$$
x^2+\log(1-x)+x= x^2/2 - x^3/3 -... \geq  0.
$$
We now set $x:=(M-m)/M$. Clearly $0<x<1/2$, $1-x=m/M$ and $m-M=-Mx$. Hence
$$
M(-\log(1-x)-x) = -M \log(m/M) + m - M
$$
and \eqref{65} implies $Mx^2 / 2 \leq - M\log(1-x) - Mx \leq Mx^2$, or,
$$
\frac{1}{2}\frac{(M-m)^2}{M} \leq  M \log(M/m) + m - M
\leq \frac{(M-m)^2}{M} . 
$$
\end{proof}

\begin{lemma}\label{log2}
If $0<m<M$ satisfies $M < 7m / 4$, then
$$
\exp\left(\frac{1}{4}\frac{(M-m)^2}{m}\right) \leq \left(\frac{m}{M}\right)^m e^{M-m} \leq
\exp\left(\frac{1}{2} \frac{(M-m)^2}{m}\right).
$$
\end{lemma}
\begin{proof}
We have for $0<x<3/4$
$$
\log(1+x)-x=-x^2/2 + x^3/3 - x^4/4 + ... \geq -x^2/2.
$$
and
$$
-x^2/4 -\log(1+x)+x = x^2/4 - x^3/3 + x^4/4 + ... \geq x^2(1/4- x/3) >0.
$$
Hence, for these $x$,
\begin{equation}
-x^2 / 2  \leq \log(1+x) -x \leq - x^2 / 4 . \label{70}
\end{equation}
Fix $0<m<M< 7m / 4$. Set $x:=M/m-1$ so that $0<x<3/4$. We have
$$
-m(\log(1+x)-x)= m \log(m/M) + M - m,
$$
hence, \eqref{70} implies
$$
\frac{m}{4} x^2 = \frac{1}{4}\frac{(M-m)^2}{m} \leq m \log(m/M) + M - m \leq \frac{1}{2} \frac{(M-m)^2}{m} = \frac{m}{2} x^2.
$$
\end{proof}

\begin{lemma}\label{estimates}
Assume that $m_n= \alpha n^2, n \geq 4,$ for some $\alpha >0$. Then
$$
\exp \alpha \leq \left(\frac{m_{n}}{m_{n+1}}\right)^{m_{n}} \exp(m_{n+1}-m_n) \leq \exp( 9\alpha/4),
$$
and
$$
\exp \alpha \leq
 \left(\frac{m_{n+1}}{m_n}\right)^{m_{n+1}} \exp(m_n - m_{n+1})
 \leq \exp(4 \alpha).
$$
\end{lemma}
\begin{proof}
Set, for $n \geq 4$, $m:=\alpha n^2$ and $M:=\alpha (n+1)^2$. Then $M-m=\alpha(2n+1)$ and $M < 7m/4 < 2m$. It is now easy to see that
$$
4 \alpha \leq \frac{(M-m)^2}{m} \leq 9 \alpha \ \ \mbox{and} \ \
2 \alpha \leq \frac{(M-m)^2}{M} \leq 4 \alpha,
$$
The conclusion follows from Lemmas \ref{log1} and \ref{log2}.

\end{proof}

\begin{proof} \textbf{of Proposition \ref{exampleexponential}.}
In this case the maximum point $r_m$ of $r^m v(r)$ is $r_m=(m/ap)^{1/p}$,  and $v(r_m)=\exp(-m/p)$ for each $m \in \N$. Therefore
$$
A(m_n,m_{n+1})^p = \left(\frac{m_{n}}{m_{n+1}}\right)^{m_{n}} \exp(m_{n+1}-m_n),
$$
and
$$
B(m_n,m_{n+1})^p = \left(\frac{m_{n+1}}{m_n}\right)^{m_{n+1}} \exp(m_n - m_{n+1}).
$$
We apply Lemma \ref{estimates} for $\alpha:=p (\log b)$ to conclude
$$
b^p=\exp(\alpha) \leq A(m_n,m_{n+1})^p \leq \exp(9\alpha / 4) = b^{9p/4} ,
$$
and
$$
b^p=\exp(\alpha) \leq B(m_n,m_{n+1})^p \leq \exp(4 \alpha) = b^{4p}.
$$
This implies the inequalities in the statement.
\end{proof}

\begin{remark} \label{rem9}
In view of Lemmas \ref{luskylemma1}--\ref{luskylemma3} it is of interest to compare
the expressions $A(m,n) $ and $B(m,n)$. Let us show here that
$$A(m_n,m_{n+1}) \leq B(m_n,m_{n+1}) \ \ (*)$$ for $v(r) = \exp(-ar^p)$ and
$m_n = \alpha n^2, \alpha >0$. It is enough to do the calculation for $m_n = n^2$. In this case we have
$$
A(m_n,m_{n+1})= \Big( \frac{n}{n+1}\Big)^{2n^2}e^{2n+1} \ \ , \ \
B(m_n,m_{n+1})=  \Big( \frac{n+1}{n}\Big)^{2(n+1)^2}e^{-2n-1},
$$
so that $(*)$ is equivalent to
$$
\frac{\exp(n+(1/2))}{\left(1 + \frac{1}{n}\right)^{n^2}} \leq \frac{\left(1 + \frac{1}{n}\right)^{(n+1)^2}}{\exp(n+(1/2))},
$$
or 
\begin{equation}
e \leq \Big(1 + \frac{1}{n}\Big)^{\frac{2n^2+2n+1}{2n+1}} =: \gamma_n.
\label{105}
\end{equation}
But $\gamma_n > \left(1 + \frac{1}{n}\right)^{n + \frac{1}{2}}=: \eta_n.$ Thus,
\eqref{105} holds true, since the sequence $(\eta_n)_n$
tends to $e$ as $n \rightarrow \infty$, and it is decreasing.
To see this last fact, we write 
\begin{equation}
\Big( \frac{\eta_n}{\eta_{n+1}} \Big)^2
= \frac{(n+1)^{4n+4}}{n^{2n+1} (n+2)^{2n+3}}
= \bigg(
\frac{n+1}{n^{\frac{2n+1}{4n+4}} (n+2)^{\frac{2n+3}{4n+4}}}
\label{107}
\bigg)^{4n+n}
\end{equation}
The logarithm of the expression in the last large parenthesis is of the form
$$
\log x - \big(r \log (x-1) + (1-r) \log (x+1) \big),
$$
where $x > 1$ and $0< r < 1$. This expression is positive, since log
is a concave function on $[0,\infty[$. Hence, \eqref{107} is larger than 1.
\end{remark}

\section{The space of multipliers $( H^\infty_v(\C), \ell^p ) $.}\label{sec4}

In this final section we show that the above results can be applied to
determine some multiplier spaces. Recall that if $A$ and $B$ are vector spaces of complex sequences containing the space of all the sequences with finitely many non-zero coordinates, then the set of multipliers from $A$ into $B$ is

$$(A,B):= \{ c=(c_n) \, : \, (c_na_n) \in B \ \forall (a_n) \in A \}. $$

Given a strictly increasing, unbounded sequence $J= (m_n)_{n=0}^\infty
\subset \N$  and  $1 \leq p,q \leq \infty$ we denote as in \cite{BlZ}, Definition 2,
\begin{equation}
\ell^J (p,q) := \Big\{ (a_m)_{m=0}^\infty \, : \, \Big( \sum_{m=m_n+1}^{m_{n+1}}
|a_m|^p \Big)^{1/p} \in \ell_q \Big\},
\end{equation}
with the obvious changes when $p$ or $q$ is $\infty$. The space $\ell^J (p,q)$
is a Banach space when endowed with  the canonically defined norm. Observe
that $\ell^J(p,p) = \ell_p$.

\begin{lemma}
\label{lem5.1}
For $1 \leq p \leq \infty$ we have
\begin{equation}
\big( \ell^J(2,\infty) , \ell^p \big) =
\ell^J(r,s)
\end{equation}
where (a) $r=2p/(2-p)$, $s=p$, if $1 \leq p < 2$, (b) $r=\infty$,
$s=p$, if $2 \leq p < \infty$, and (c) $r= s=\infty$, if $p = \infty$.
\end{lemma}

\begin{proof}
This is a direct consequence of \cite{BlZ}, Theorem 23. As it is mentioned in that
paper, which treats more complicated cases, the proof in the case of our Lemma is similar to that of \cite{K}, Theorem 1.
\end{proof}

\begin{theorem}
\label{th5.2}
Let $v$ be a radial weight and let  $J= (m_n)_{n=0}^\infty $ be a strictly increasing, unbounded sequence of positive integers such that for some $b > 2$ and $K \geq b$ we have
\begin{equation}
b \leq \min(A(m_n,m_{n+1}),B(m_n,m_{n+1})) \leq \max(A(m_n,m_{n+1}),
B(m_n,m_{n+1})) \leq K.
\label{135}
\end{equation}
Let $ 1 \leq p \leq  \infty$. Then $(\lambda_m)_{m=1}^\infty$ is a multiplier
from $H_v^\infty(\C)$ into $\ell_p$ if and only if
\begin{equation}
\Big( \big( (v(r_{m_n}) r_{m_n}^m )^{-1} |\lambda_m| \big)_{m=m_n+1}^{m_{n+1}}
\Big)_{n=0}^\infty \in \ell^J(r,s) ,
\end{equation}
where

\noindent (a) $r=2p/(2-p)$, $s=p$, if $1 \leq p < 2$,

\noindent (b) $r=\infty$, $s=p$, if $2 \leq p < \infty$, and

\noindent (c) $r= s=\infty$, if $p = \infty$.

\end{theorem}

\begin{proof}
Since $\ell^p$ is a solid space, we can apply Fact 5 in the introduction to
conclude
\begin{equation}
\big( H_v^\infty( \C) , \ell_p \big)=
\big( S( H_v^\infty( \C)) , \ell_p \big) .
\end{equation}
By Theorem \ref{solidhull} it is easy to see that
$(\lambda_m)_{m=0}^\infty \in \big(  S( H_v^\infty( \C)) , \ell_p \big)$, if and
only if
\begin{equation}
\Big( \big( (v(r_{m_n}) r_{m_n}^m )^{-1} |\lambda_m| \big)_{m=m_n+1}^{m_{n+1}}
\Big)_{n=0}^\infty \in \big( \ell^J(2, \infty) , \ell_p\big) .
\end{equation}
The conclusion now follows from Lemma \ref{lem5.1}
\end{proof}

The next corollary is a consequence of Theorem \ref{th5.2} and
Proposition \ref{exampleexponential}.

\begin{corollary}
Let $v(r) = e^{-r}$, $r \in (0,\infty)$ and $1 \leq p \leq \infty$.
Then, the space of multipliers $\big( H_v^\infty(\C), \ell_p \big)$ is the set of sequences
$(\lambda_m)_{m=0}^\infty$ such that
\begin{equation}
\Big( \sum_{n=1}^\infty  \Big( \sum_{m=n^2 +1}^{(n+1)^2}
\big( |\lambda_m| e^{-n^2} n^{-2m}  \big)^{\frac{2p}{2-p}}
\Big)^{\frac{2-p}{2}} \Big)^{\frac{1}{p}} < \infty  ,
\end{equation}
if $1 \leq p < 2$,
\begin{equation}
\Big( \sum_{n=1}^\infty  \Big( \max\limits_{n^2 < m \leq (n+1)^2}
 |\lambda_m| e^{n^2} n^{-2m}  \Big)^{p} \Big)^{\frac{1}{p}} < \infty  ,
\end{equation}
if $2 \leq p < \infty$, and
\begin{equation}
\sup\limits_{n \in \N } \Big( \max\limits_{n^2 < m \leq (n+1)^2}
 |\lambda_m| e^{n^2} n^{-2m} \Big) < \infty  ,
\end{equation}
if $p =  \infty$,

\end{corollary}

\vspace{.5cm}

\noindent \textbf{Acknowledgements.} The authors are indebted to the referee for the careful reading of the manuscript and the helpful suggestions.

The research of Bonet was partially supported by MTM2013-43540-P, GVA Prometeo II/2013/013 and GVA ACOMP/2015/186. The research of Taskinen was partially supported by the Magnus Ehrnrooth and the V\"ais\"al\"a
Foundations.

\noindent \textbf{Authors' addresses:}%
\vspace{\baselineskip}%

Jos\'e Bonet: Instituto Universitario de Matem\'{a}tica Pura y Aplicada IUMPA,
Universitat Polit\`{e}cnica de Val\`{e}ncia,  E-46071 Valencia, Spain

email: jbonet@mat.upv.es \\

Jari Taskinen: Department of Mathematics and Statistics, P.O. Box 68,
University of Helsinki, 00014 Helsinki, Finland.

email: jari.taskinen@helsinki.fi

\end{document}